\documentclass{amsart}[11pt]
\usepackage{amsmath,amsfonts,amssymb,latexsym,amsthm}
\usepackage{epsfig}
\usepackage{psfrag}
\usepackage{hyperref} 
\usepackage[dvipsnames]{xcolor}

\usepackage{hyperref}

\usepackage[all]{xy}

\newtheorem{thm}{Theorem}[section]
\newtheorem{lemma}[thm]{Lemma}

\newtheorem{cor}[thm]{Corollary}

\newtheorem{question}[thm]{Question}

\theoremstyle{plain}

\theoremstyle{definition}
\newtheorem{Example}[thm]{Example}

\newtheorem{rem}[thm]{Remark}

\numberwithin{equation}{section}

\def\emp{\nothing}
\def\sq{\square}

\def\nn{\mathbb N}
\def\cc{\mathbb C}
\def\rr{\mathbb R}

\def\Ga{\Gamma}

\def\De{\Delta}

\def\la{\lambda}
\def\ga{\gamma}

\def\cC{\mathcal C}
\def\cB{\mathcal B}
\def\ca{\mathcal A}
\def\cA{\mathcal A}
\def\cb{\mathcal B}

\def\cR{\mathcal R}

\def\ssu{\subset}

\def\<{\langle}
\def\>{\rangle}

\def\rR{ {\text {\rm R} } }

\def\SYT{ {\text {\rm SYT} } }

\def\Ups{\Upsilon}

\def\rH{{\text {\sc H} } }

\def\0{{\mathbf 0}}

\def\nothing{\varnothing}

\def\.{\hskip.06cm}
\def\ts{\hskip.03cm}

\def\lra{\leftrightarrow}

\def\bx{\textbf{\textit{x}}}
\def\bc{\textbf{\textit{c}}}

\def\by{\textbf{\textit{y}}}

\def\bba{\textbf{\textit{a}}}
\def\bbb{\textbf{\textit{b}}}

\def\sign{{\rm sign}}

\def\nin{\noindent}

\renewcommand{\mod}[1]{
\;\, \textup{mod} \; #1
}

\def\lra{\leftrightarrow}

\def\wh{\widehat}

\title[Hidden symmetries]{Hidden symmetries of weighted lozenge tilings}
\date{}

\author[Igor Pak \and Fedor Petrov]{Igor~Pak$^{\star}$  \  and  \   Fedor Petrov$^{\dagger}$}

\thanks{\thinspace ${\hspace{-.45ex}}^\star$Department of Mathematics,
UCLA, Los Angeles, CA, 90095. \,  Email: \texttt{pak@math.ucla.edu}}

\thanks{\thinspace ${\hspace{-.45ex}}^\dagger$Steklov Mathematical Institute,
St.\ Petersburg, Russia. \,  Email: \texttt{fedyapetrov@gmail.com}}

\thanks{\thinspace \
\today}

\begin{document}
\maketitle

\begin{abstract}
We study the weighted partition function for
lozenge tilings, with weights given by multivariate rational functions
originally defined in~\cite{MPP3} in the context of the
factorial Schur functions.  We prove that this partition
function is symmetric for large families of regions.  We employ
both combinatorial and algebraic proofs.
\end{abstract}


\section{Introduction}
\emph{Hidden symmetries} are pervasive across the natural sciences,
but are always a delight whenever discovered.  In Combinatorics, they
are especially fascinating, as they point towards both advantages and
limitations of the tools, cf.~$\S$\ref{ss:finrem-other}.
Roughly speaking, the combinatorial approach strips away much
of the structure, be it algebraic, geometric, etc., while
allowing a direct investigation often resulting in an explicit
resolution of a problem.  But this process comes at a cost ---
when the underlying structure is lost, some symmetries become
invisible, or ``hidden''.

Occasionally this process runs in reverse.  When a hidden symmetry
is discovered for a well-known combinatorial structure, it is as
surprising as it is puzzling, since this points to a rich structure
which yet to be understood (sometimes uncovered many years later).
This is the situation of this paper.

\smallskip

We enumerate the (weighted) lozenge tilings of regions on a triangular
lattice.  These tiling problems appear in a number of interrelated areas:
from general tiling literature~\cite{Thu} to combinatorics of plane
partitions~\cite{Kra}, to statistical physics of the dimer model~\cite{Gor}.
First studied by MacMahon, Kasteleyn and Temperley--Fisher in other settings,
these lozenge tilings are now extremely well understood by tools of the
determinant calculus, algebraic combinatorics and integrable
probability (see Section~\ref{s:finrem}).  Yet our hidden symmetries
appear to be new (see, however, $\S$\ref{ss:finrem-BP}).

\smallskip

The results of this paper are somewhat technical, but the
backstory is quite interesting.  We start with a classical result
of MacMahon: the number \ts $P_{a b c}$ \ts of plane partitions which
fit into \ts $[a\times b \times c]$ \ts box is given by a product
formula:
\begin{equation} \label{eq:macmahon}
P_{a b c} \, = \, \prod_{i=1}^a\prod_{j=1}^{b}\prod_{k=1}^c \.
\frac{i+j+k-1}{i+j+k-2}\,,
\end{equation}
%
%
%
If you think of these \emph{boxed plane partitions} as 3-dimensional
objects and squint your eyes, you see that they are in natural
bijection with lozenge tilings of the \ts
$\<a\times b \times c \times a \times b \times c\>$ \ts hexagon~$\rH\<a,b,c\>$,
see Figure~\ref{f:MPP3-Hex-big}.

There are numerous extensions and generalizations of~\eqref{eq:macmahon},
and it is key to many recent probabilistic studies.  On a combinatorial side,
there is a classical $q$-analogue \ts $P_{a b c}(q)$ \ts by the ``volume''
of the tilings, which corresponds to the size of the
plane partition.  If one views~\eqref{eq:macmahon} as an evaluation
of the Schur function, this $q$-analogue is given by
$$P_{a b c}(q) \. = \. q^{-a(a+1)b/2} \ts \cdot \ts s_{(b^a)}\bigl(1,q,\ldots,q^{c-1}\bigr)\ts.
$$
\begin{figure}[hbt]
\begin{center}
\psfrag{F}{{$\Phi$}}
\psfrag{H}{{\hskip-.3cm {\small $\rH\<5,11,4\>$}}}
\psfrag{P}{{\small \, $[9\times 12]$}}
\epsfig{file=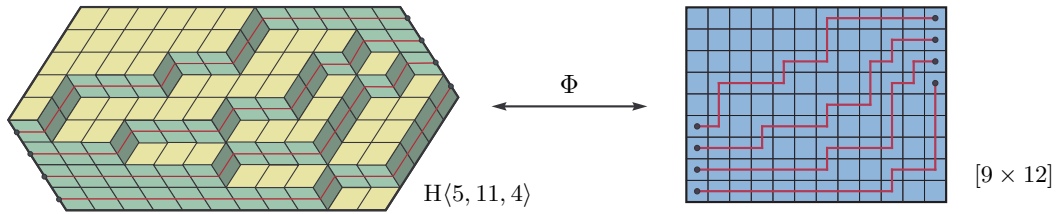,width=13.2cm}
\end{center}
\caption{A lozenge tiling of \ts $\rH\<5,11,4\>$ \ts and the
corresponding collection of non-intersecting paths in $[9\times 12]$.}
\label{f:MPP3-Hex-big}
\end{figure}

%
When the bottom rectangle $(b^a)$ is replaced by a Young diagram~$\la$,
there is \emph{Stanley's hook-content formula} for $s_\la(1,q,\ldots,q^{c-1})$.
There are many other exact product formulas for various further extensions,
some related to other root systems and symmetry classes recently surveyed
in~\cite{Kra}, some with surprising coincidences and hidden symmetries~\cite{Ste}.

On a probabilistic side, there is a celebrated \emph{Arctic circle} phenomenon
first discovered in~\cite{CLP} for $\rH\<n,n,n\>$, and then extended to general
regions in~\cite{CKP}. This work led to an incredible wealth of results
on the \emph{limit shapes} and \emph{random surfaces}, most of which goes
outside the scope of this paper, see an extensive survey~\cite{Gor}.
Let us single out~\cite{BGR} which gives a $5$-parameter elliptic deformations
(with one relation) of $P_{a b c}(q)$, and computed the exact asymptotic
formulas for the limit shape.

\smallskip

Our approach to a multivariate deformation of \ts $P_{a b c}$ \ts
is based on the recent work~\cite{MPP3} in Algebraic Combinatorics,
in turn inspired by the extensive study of the (equivariant)
cohomology of the Grassmannian.  To set this up, recall that
the lozenge tilings of \ts $\rH\<a,b,c\>$ \ts
are in bijection with collections of non-intersecting paths
in the rectangle, see Figure~\ref{f:MPP3-Hex-big}.  These
lattice paths are in bijection with the \emph{excited
diagrams}, thus giving a connection to the \emph{Naruse hook-length
formula}~\cite{MPP1,MPP2} the number of standard Young tableaux
of skew shapes.

In~\cite{MPP3}, the authors introduce a multivariate
deformation \ts $F_{a b c}(x_1,x_2,\ldots \, | \, y_1,y_2, \ldots )$
\ts of \ts $P_{a b c}$ \ts with two sets of variables
which play a superficially similar role:
$$F_{a b c} (1,1,\ldots  \, | \, 0, 0, \ldots) \. = \.
F_{a b c} (0,0,\ldots  \, | \, 1, 1, \ldots) \. = \. P_{a\ts (b-1)\ts c}
\..
$$
The key technical result in~\cite{MPP3} is the symmetry of \ts $F_{a b c}$.

Formally, the Morales--Pak--Panova (MPP--) Theorem~\ref{t:MPP3-identity},
shows that \ts $F_{a b c}(\bx\ts | \ts \by )$ \ts
is symmetric in the first set of variables~$\bx=(x_1,x_2,\ldots)$,
with the second set \ts
$\by=( y_1,y_2, \ldots )$ \ts as parameters, and vice versa (see~$\S$\ref{ss:finrem-MPP3}).
This result is derived from the algebraic properties of the
\emph{factorial symmetric functions}
defined by Macdonald in one of his ``variations"~\cite{Mac}.
These symmetric functions were later studied by Molev--Sagan~\cite{MS},
Ikeda--Naruse~\cite{IN}, and others, in connection
with the \emph{equivariant Schubert calculus}.  The authors
use a special case of this hidden symmetry to give product formulas
for the number of standard Young tableaux \ts $\SYT(\la/\mu)$,
for a $6$-parameter family \ts $\{\la/\mu\}$ \ts of skew Young diagrams.

\medskip

\nin
We obtain two generalizations and refinements of the MPP--theorem, to:
\smallskip

\nin
\hskip.5cm $(1)$ \ts \emph{trapezoid} (sawtooth) \emph{regions} obtained from \ts $\rH(a,b,c)$
by horizontal cuts,


\nin
\hskip.5cm  $(2)$ \ts \emph{parallelogram regions} obtained from \ts $\rH(a,b,c)$
by two vertical cuts.

\smallskip

\nin
Formally, for general regions~$\Ga$, we define a multivariate partition
function \ts $F(\bx \ts | \ts \by)$ \ts by summing over all lozenge tilings   
of~$\Ga$.  In case~$(1)$, we show that \ts $F(\bx \ts | \ts \by)$ \ts
is symmetric in~$\bx$, and in case~$(2)$ we show that
\ts $F(\bx \ts | \ts \by)$ \ts is symmetric in~$\by$.  Both
results generalize (two parts of) the MPP--theorem, which until now had
only a technical proof based on the properties of factorial Schur
functions.  We then obtain a common generalization  \ts
Main Theorem~\ref{t:main-gen}.
We leave open the problem of finding probabilistic
and enumerative applications of these general hidden symmetries.

\smallskip

The rest of the paper is structured as follows.  We start by
stating both the background and the results in Section~\ref{s:main},
followed by their lozenge tilings interpretation and quick pointers
to the literature.  In the following two sections (Section~\ref{s:comb}
and~\ref{s:alg}), we give
completely independent combinatorial and algebraic proofs of
the results, including the proof of Main Theorem~\ref{t:main-gen}.
We conclude with final remarks and open problems
in Section~\ref{s:finrem}.

\bigskip

\section{Main results}\label{s:main}

\subsection{Known results} \label{ss:main-known}

We start with the MPP--theorem mentioned in the introduction:

\smallskip

\begin{thm}[{Morales--Pak--Panova~\cite[Thm~3.10]{MPP3}}]\label{t:MPP3-identity}
Define the following multivariate rational function:
\begin{equation}
\label{eq:Naruse-1-path}
F_{a b c}\bigl(x_1,\ldots,x_{a+c} \, | \, y_1,\ldots,y_{b+c}\bigr) \. := \,
\sum_{\substack{\Ups=(\ga_1,\ldots,\ga_c) \\ \ga_k\ts{}:\ts{}(a+k,1)\to
  (k,b+c)}} \prod_{k=1}^c \. \prod_{(i,j) \in \ga_k} \. \frac{1}{x_i+y_j}\.,
 \end{equation}
where the sum is over all collections~$\Ups$ of non-intersecting lattice paths in
the \ts $[(a+c)\times (b+c)]$ \ts rectangle $($see Figure~\ref{f:MPP3-big}$)$.
Then \ts $F_{a b c}(\bx \. | \ts \by \ts)$ \ts
is symmetric in \ts $\bx=(x_1,\ldots,x_{a+c})$ \ts and in \ts $\by=(y_1,\ldots,y_{b+c})$.
\end{thm}

Strictly speaking, Theorem \ref{t:MPP3-identity} follows from the proof
of Thm~3.10 in \cite{MPP3}, but not from the statement. 

\begin{figure}[hbt]
\begin{center}
\psfrag{x}{$x$}
\psfrag{y}{$y$}
\psfrag{x1}{$x_1$}
\psfrag{x2}{$x_2$}
\psfrag{y1}{$y_1$}
\psfrag{y2}{$y_2$}
\psfrag{Y1}{$\Ups_0$}
\psfrag{Y2}{$\Ups_1$}
\psfrag{Y3}{$\Ups_2$}
\psfrag{A1}{\small $A_1$}
\psfrag{A2}{\small $A_2$}
\psfrag{A3}{\small $A_3$}
\psfrag{A4}{\small $A_4$}
\psfrag{B1}{\small $B_1$}
\psfrag{B2}{\small $B_2$}
\psfrag{B3}{\small $B_3$}
\psfrag{B4}{\small $B_4$}
\epsfig{file=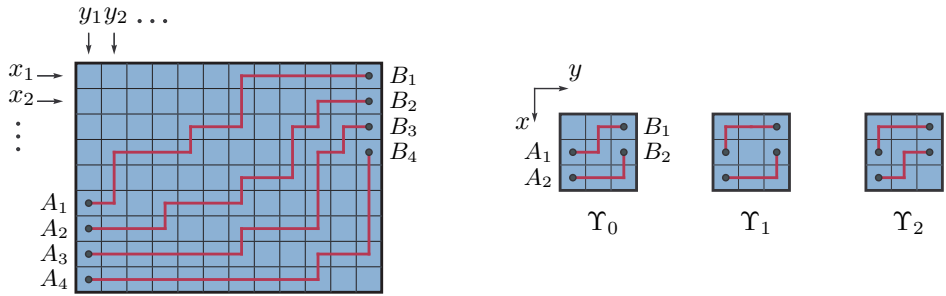,width=12.8cm}
\end{center}
\caption{\underline{Left}: An example of a collection $\Ups$ of $c$
paths as in Theorem~\ref{t:MPP3-identity},
where \ts $a=5$, $b=8$, and $c=4$. \underline{Right}:
An example of all three possible paths $\Ups_0, \Ups_1,\Ups_2$
for $a=1$, $b=1$, and $c=2$.}
\label{f:MPP3-big}
\end{figure}


Here and everywhere below we adopt the coordinate system that is standard
for matrices: the first coordinate~$x$ is increasing downwards and the
second coordinate~$y$ is increasing from left to right
(see Figure~\ref{f:MPP3-big}).

\begin{Example}{\rm Let $a=1$, $b=1$ and $c=2$.  We have $A_1=(2,1)$, $A_2=(3,1)$,
$B_1=(1,3)$ and $B_2=(2,3)$, and \ts $\Ups = (\ga_1,\ga_2)$ \ts
are non-intersecting paths \ts $\ga_1: A_1 \to B_1$ \ts and \ts $\ga_2: A_2\to B_2$
inside a $3\times 3$ square.
There are three such $\Ups$ avoiding either $(1,1)$, or $(2,2)$,
or~$(3,3)$, see Figure~\ref{f:MPP3-big} (Right).   For example, $\Ups_0=(\ga_1,\ga_2)$,
where \ts $\ga_1: A_1=(2,1)\to (2,2) \to (1,2) \to (1,3)=B_1$, and
\ts $\ga_2: A_2=(3,1)\to (3,2) \to (3,3) \to (2,3)=B_2$, i.e., $\Ups_0$ is avoiding~$(1,1)$.
We have:
$$
F_{1\ts 3\ts 2}(\bx \ts | \ts \by) \. = \. w(\Ups_0) \ts + \ts w(\Ups_1)\ts + \ts w(\Ups_2)
\. = \. \Bigl[(x_1+y_1) \. + \. (x_2+y_2)\. + \. (x_3+y_3)\Bigr]
\. \prod_{i=1}^3 \. \prod_{j=1}^3 \.\frac{1}{x_i+y_j}\.,
$$
which is symmetric in~$\bx$ and in~$\by$ (but not in \emph{both}~$\bx$ and~$\by$).
}\end{Example}

Let us emphasize that although the symmetry in both sets of variables
may seem to play the same role, the result is not symmetric under
the transposition giving \ts $\bx \lra \by$.
In fact, these are fundamentally different symmetries: the one in~$\bx$ is both
difficult and interesting, while the one in~$\by$ is relatively straightforward.
As we mentioned in the introduction, the two generalizations we present each
retain only one of these symmetries.

\subsection{New results} \label{ss:main-lattice}
There is a natural way to generalize the setting of
Theorem~\ref{t:MPP3-identity}. Let \ts
$[m\times n]=\{(p,q)\in \nn^2, 1\le p \le m, 1\le  q  \le n\}$,
$\cA=(A_1,\ldots,A_k)$, $\cB=(B_1,\ldots,B_k)$ be two
$k$-tuples of points in $[m\times n]$.
Denote by \ts $\Ups:\cA\to \cB$ a collection \ts $(\ga_1,\ldots,\ga_k)$
\ts of non-intersecting lattice paths $\ga_i: A_i \to B_i$,
and let \ts $N(\cA,\cB):=\#\{\Ups: \cA\to \cB\}$ \ts be the number
of such collections. Throughout the paper, unless stated otherwise,
all paths will use only {\tt Up} and {\tt Right} steps, where the
coordinates are arranged as in Figure~\ref{f:MPP3-big}
(see also~$\S$\ref{ss:finrem-UR}).

Note that for fixed $\cA,\cB\ssu \nn^2$, the set \ts $\{\Ups: \cA\to \cB\}$
\ts is a classical combinatorial object which generalizes
Dyck paths, plane partitions, Young tableaux, etc.~\cite[Ch.~5]{GJ}.
Under mild conditions, the number $N(\cA,\cB)$ has a determinant formula
via the \emph{Lindstr\"om--Gessel--Viennot} (LGV--)  \emph{lemma}
(see $\S$\ref{ss:alg-LGV}).
As we discussed in the introduction, for $\ca,\cb$ as in
Theorem~\ref{t:MPP3-identity}, the number \ts $N(\ca,\cb)$ \ts of
non-intersecting collections of paths \ts $\Ups: \cA\to \cB$
is equal to \ts $P_{a b c}$ given by~\eqref{eq:macmahon}.

Define the \emph{weight} of~$\Ups$ as
$$
w(\Ups) \. := \. \prod_{i=1}^k \. w(\ga_i)\., \quad \text{where} \quad
w(\ga) \. : = \.  \prod_{(i,j)\in \ga} \. \frac{1}{x_i+y_j}\,.
$$
Let
\begin{equation}\label{eq:F-def}
F_{\cA,\cB}\bigl(x_1,\ldots,x_m\, | \, y_1,\ldots,y_n\bigr)\, := \,
\sum_{\Ups: \ts\cA\to\cB} \. w(\Ups)\ts.
\end{equation}
Note that $F$ is not symmetric for general $\cA,\cB$.  For example,
let \ts $k=2$, \ts $A_1=B_1=(1,1)$, \ts $A_2=B_2=(2,2)$. Then \ts
$N(\cA,\cB)=1$, and
$$
F_{\cA,\cB}(x_1,x_2\, | \, y_1,y_2) \, = \, \frac{1}{(x_1+y_1)(x_2+y_2)}\,,
$$
which is not symmetric in either set of variables. Since there is
no apparent action of either symmetric group on the paths
collections~$\Ups$ in Theorem~\ref{t:MPP3-identity}, the theorem
represents a \emph{hidden symmetry}, and raises the following general
question:

\begin{question} Are there other sets \ts $\cA,\ts \cB\ssu [m\times n]$,
for which the multivariate generating function \ts $F_{\cA, \cB}(\bx \. | \ts \by)$ \ts
is symmetric in \ts $(x_1,\ldots,x_m)$?
\end{question}

We give two positive answers to this question, refining both symmetries in
Theorem~\ref{t:MPP3-identity}~:

\begin{thm}[Horizontal cut] \label{t:hor-cut}
Let \ts $m=a+k$, \ts $A_1=(a+1,1)$, \ldots, \ts $A_k=(m,1)$,
and $\cA=(A_1,\ldots,A_k)$. Similarly, let \ts $B_1=(1,b_1)$,
 \ldots, \ts $B_k=(1,b_k)$, for some \ts
 $1\le b_1 < b_2 < \ldots < b_k \le n$, and \ts $\cB=(B_1,\ldots,B_k)$.
 Then the multivariate function
$$
F_{\ca,\cb}\bigl(x_1,\ldots,x_m\, | \, y_1,\ldots,y_n\bigr)
$$
defined in~\eqref{eq:F-def}, is symmetric in \ts $\bx=(x_1,\ldots,x_{m})$.
\end{thm}
\begin{figure}[hbt]
\begin{center}
\psfrag{A1}{\small $A_1$}
\psfrag{A2}{\small $A_2$}
\psfrag{A3}{\small $A_3$}
\psfrag{A4}{\small $A_4$}
\psfrag{B1}{\small $B_1$}
\psfrag{B2}{\small $B_2$}
\psfrag{B3}{\small $B_3$}
\psfrag{B4}{\small $B_4$}
\epsfig{file=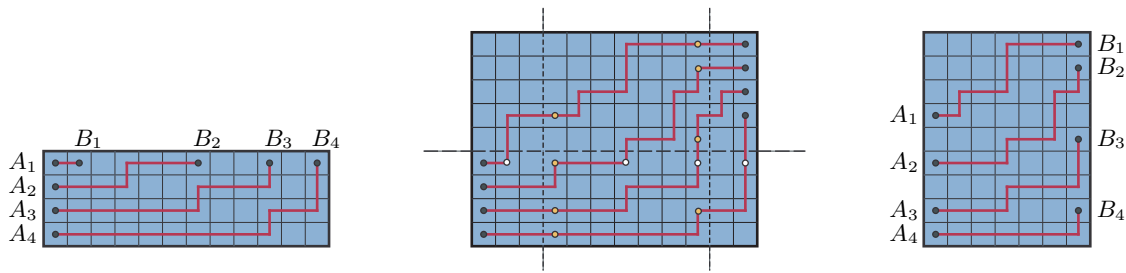,width=15.2cm}
\end{center}
\caption{Examples of paths collections in
Theorems~\ref{t:hor-cut} and~\ref{t:vert-cut} and how they refine
Theorem~\ref{t:MPP3-identity}. }
\label{f:two-cuts}
\end{figure}

\smallskip

See Figure~\ref{f:two-cuts} for the explanation of the
\emph{horizontal cut} in the title.  Let us show that
Theorem~\ref{t:hor-cut} implies the $\bx$-symmetry part of
Theorem~\ref{t:MPP3-identity}, for $a\ge c$. 
Apply Theorem~\ref{t:hor-cut} to two adjacent cuts: above and below
row~$(a+1)$, including the latter into both parts.
Of course, to apply Theorem~\ref{t:hor-cut} to the top part,
rotate it 180~degrees.  We obtain that $F_{abc}$ is symmetric
in both \ts $(x_1,\ldots,x_{a+1})$ \ts and in \ts \ts
$(x_{a+1},\ldots,x_{a+c})$, implying the symmetry in
\ts $(x_1,\ldots,x_{a+c})$, for every fixed start/end points $\cC$
in row $(a+1)$.  Summing over all such~$\cC$, we obtain
Theorem~\ref{t:MPP3-identity}.

\smallskip

\begin{thm}[Vertical double cut] \label{t:vert-cut}
Let \ts $A_1=(a_1,1)$, \ldots, \ts $A_k=(a_k,1)$, for some \ts
 $1\le a_1 < a_2 < \ldots < a_k \le k+\ell$, and \ts $\cA=(A_1,\ldots,A_k)$.
 Similarly, let \ts $m\ge 1$, \ts
  $B_1=(b_1,m)$, \ldots, \ts $B_k=(b_k,m)$, for some \ts
 $1\le b_1 < b_2 < \ldots < b_k \le k+\ell$, and \ts $\cB=(B_1,\ldots,B_k)$.
Then the multivariate function
$$
F_{\ca,\cb}\bigl(x_1,\ldots,x_{k+\ell}\, | \, y_1,\ldots,y_m\bigr),
$$
defined in~\eqref{eq:F-def}, is symmetric in \ts $\by=(y_1,\ldots,y_{m})$.
\end{thm}

\smallskip

In the theorem, one can assume that $a_i \ge b_i$ for all $i=1,\ldots, k$,
since otherwise there are no collections of \ts {\tt Up-Right}~paths~$\Ups$,
and the claim is vacuously true (cf.~$\S$\ref{ss:finrem-UR}).
We should mention that this generalization of the $\by$-symmetry
part of Theorem~\ref{t:MPP3-identity} is conceptually more straightforward,
as it both contains it as a special case and refines it, see Figure~\ref{f:two-cuts}.

\begin{rem} Darij Grinberg (private communication)
 suggested the following way to deduce Theorem \ref{t:hor-cut}
from Theorem \ref{t:vert-cut}. Denote $C_i=(m,i)$ for $i=1,\ldots,k$, $\cC=(C_1,\ldots,C_k)$.
There is a natural weight-preserving
bijection between collections of paths $\Ups_a: \ts\cA\to\cB$ and $\Ups_c: \ts\cC\to\cB$:
replace the horizontal initial segments $[A_i,(a+i,i)]$ in $\Ups_a$
to the vertical initial segments $[C_i,(a+i,i)]$ in $\Ups_c$.
\end{rem}

\medskip

\subsection{Lozenge tilings formulation} \label{ss:main-lozenge}
Let us recall the bijection~$\Phi$ in Figure~\ref{f:MPP3-Hex-big}
which allows us to translate the lattice paths results into
statements about lozenge tilings.  Start with \ts
$\Ups =(\ga_1,\ldots,\ga_c)$ in the rectangle \ts $S:=[(a+c)\times (b+c)]$.
Place points in the middle of edges of the opposite $c$ edges
in \ts $\rH=\rH\<a,b-1,c\>$ \ts as in the Figure~\ref{f:MPP3-Hex-big}.  Think of paths
$\ga_i$ in $S$ in  as a union of edges.  Start with vertices in the lower
left edge of~$\rH$ as in the Figure.  For every  {\tt Right}
edge in $\ga_i$, make a {\tt Right} edge through a light green lozenge
in~$\rH$.  Similarly, for every {\tt Up} edge in $\ga_i$ make a
{\tt Up-Right} (diagonal) edge through a dark green lozenge in~$\rH$.
When all of $\Ups$ is mapped onto~$\rH$, we obtain a partial tiling
of the hexagon with light and dark green lozenges.  Fill the
remaining space with yellow lozenges. This completes the
construction of~$\Phi$.

We refer to~\cite[$\S$7]{MPP3}
for more details and properties of this bijection,
reformulation of Theorem~\ref{t:MPP3-identity} into
the lozenge language and several applications.
We should also mention that our
deformation \ts $F_{abc}(\bx\ts | \ts \by)$ \ts for $x_i=q^{i}$,
$y_j = -q^{-j}$, is well-known
as a $q$-Racah special case studied in~\cite{BGR},
see~\cite[$\S$9.6]{MPP3} for a detailed explanation.

\smallskip
Now, consider the \emph{trapezoid} (sawtooth)
\emph{region} \ts $\Ga(c_1,\ldots,c_k)$ \ts defined as in Figure~\ref{f:Sawtooth}.
\begin{figure}[hbt]
\psfrag{c1}{$c_1$}
\psfrag{c2}{$c_2$}
\psfrag{c3}{$c_3$}
\psfrag{c4}{$c_4$}
\psfrag{x1}{$x_1$}
\psfrag{x2}{$x_2$}
\psfrag{y1}{$y_1$}
\psfrag{y2}{$y_2$}
\psfrag{G}{$\Ga$}
\begin{center}
\epsfig{file=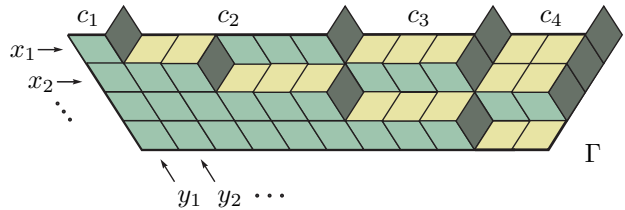,width=8.5cm}
\end{center}
\caption{Lozenge tiling of a trapezoid region \ts
$\Ga=\Ga(1,5,3,2)$ \ts for \ts $k=4$. }
\label{f:Sawtooth}
\end{figure}
This region corresponds to Theorem~\ref{t:hor-cut} with $a=0$ and
$b_1=1+c_1$, $b_2=1+c_1+c_2$, \ldots, $b_k=1+c_1+\ldots + c_k$.
For the example in Figure~\ref{f:Sawtooth} the region $\Ga(1,5,3,2)$ \ts
corresponds to \ts $\ca=\{(1,1), (1,2), (1,3), (1,4)\}$ \ts and
\ts $\cb=\{(1,2), (1,7) , (1,10), (1,12)\}$,
as in Figure~\ref{f:two-cuts} (left).

In fact, the lozenge tilings of regions $\Ga(\bc)$ are heavily
studied in integrable probability, see~\cite{Nov,Pet}.  The total
number $N(\la)$ of such tilings is given by the formula
$$
N(\la) \, = \, s_\la(1^k) \, = \, \prod_{1\le i < j\le k} \.
\frac{b_j-b_i}{j-i}\,,
$$
where \ts $\la=(\la_1,\ldots,\la_k)$, and \ts $\la_i=b_{k+1-i}-k+i$ \ts
for all $1\le i\le k$. We refer to~\cite[$\S$19]{Gor} for an interesting
discussion of this special case, further results and references.

Theorem~\ref{t:hor-cut} thus gives a multivariate deformation
of \ts $N(\la)$.  The weights \ts $1/(x_i+y_j)$ \ts are assigned
to light green lozenges and bottom halves of dark green lozenges
as shown in Figure~\ref{f:Sawtooth}.  Yellow lozenges get weight~1.
The weight of a tiling is then a product of weights of all lozenges.
The resulting partition function is then the sum of all weights
of lozenge tilings of fixed~$\Ga$ as above.  By
Theorem~\ref{t:hor-cut}, this function is symmetric.

\begin{figure}[hbt]
\psfrag{a}{$\bba$}
\psfrag{b}{$\bbb$}
\psfrag{D}{$\De$}
\psfrag{m}{$m$}
\psfrag{x1}{$x_1$}
\psfrag{x2}{$x_2$}
\psfrag{y1}{$y_1$}
\psfrag{y2}{$y_2$}
\begin{center}
\epsfig{file=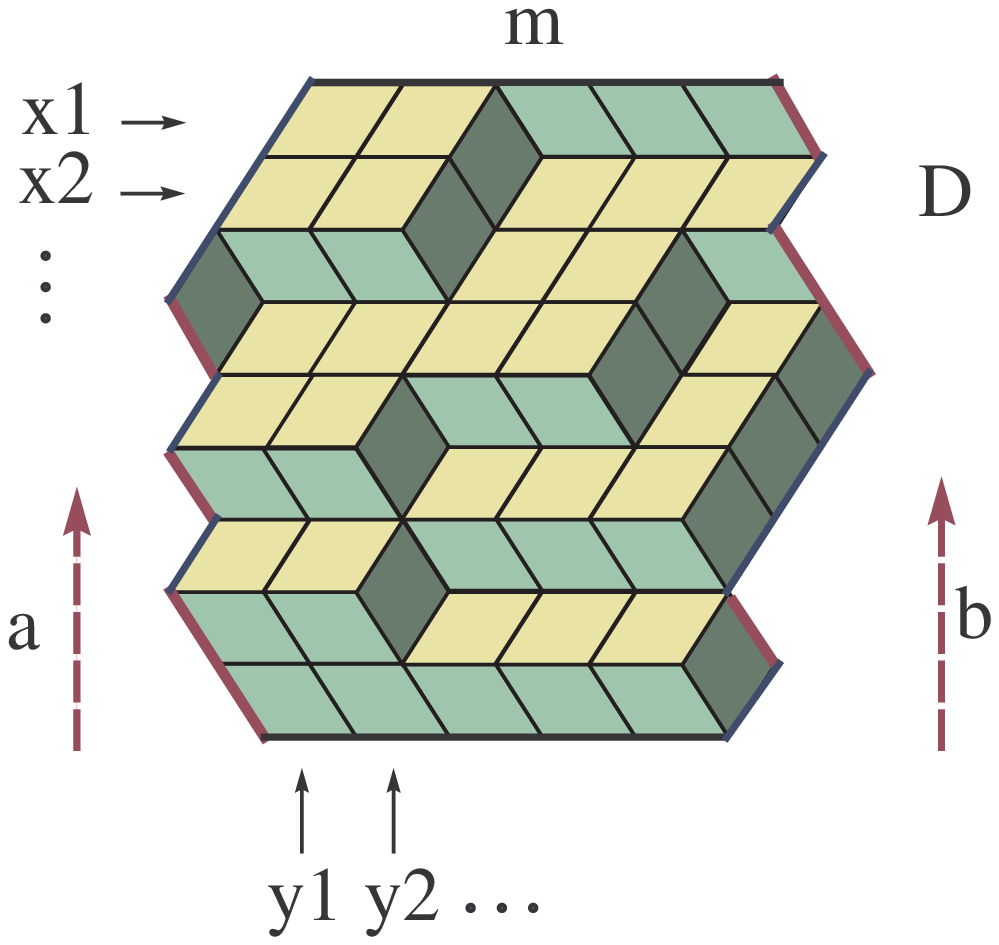,width=5.1cm}
\end{center}
\caption{Lozenge tiling of a parallelogram region \ts
$\De=\De(\bba,\bbb,m)$, for \ts $k=4$, $\ell=m=5$. }
\label{f:parallel}
\end{figure}

Note that in every simply connected region tileable with
lozenges, the boundary has $2k, 2\ell$ and~$2m$ edges
in each of the three directions.  A \emph{parallelogram region} $\De$
is defined to have two intervals of $m$ consecutive (say, horizontal)
edges.  This condition automatically implies that between the horizontal
edges there are $(k+\ell)$ edges on each side, see Figure~\ref{f:parallel}.
The region is thus encoded $\De=\De(\bba,\bbb,m)$ by two increasing sequences \ts
$\bba=(a_1,\ldots,a_k)$ \ts and \ts $\bbb=(b_1,\ldots,b_k)$, where
\ts $1\le a_i, b_i\le k+\ell$ \ts.
For example, for the region in the figure, we have $k=4$, $\ell=m=5$, and
the sequences are \ts $\bba=(1,2,3,5)$, and \ts $\bbb=(2,6,7,9)$.

In these notation, Theorem~\ref{t:vert-cut} proves the $\bx$-symmetry of the
multivariate deformation of the number \ts $N(\bba,\bbb,m)$ \ts of tilings
of a parallelogram region \ts $\De(\bba,\bbb,m)$ \ts defined above.
Here the weighting is similar to the previous case but somewhat more awkward,
see Figure~\ref{f:parallel}.  While
we do not know (or do not recognize) the number \ts $N(\bba,\bbb,m)$,
let us mention that it has a determinant formula via the LGV--lemma,
which is also the key to the proof of Theorem~\ref{t:vert-cut}.

\bigskip

\section{Combinatorial proofs}\label{s:comb}

\subsection{The $2$-symmetry case} \label{s:comb-2sym}
We start with a special case $a=c=1$ in Theorem~\ref{t:MPP3-identity}
(cf.~$\S$\ref{ss:finrem-2sym}).

\begin{lemma}[{$2$-symmetry}] \label{l:2sym}
Let \ts $A=(2,1)$, $B=(1,m)$, $m\ge 1$.  Let
$$F_{m}(x_1,x_2\. | \. y_1,\ldots,y_m) \. := \. \sum_{\ga: \. A\to B}
\. \prod_{(i,j)\in \ga} \. \frac{1}{x_i+y_j}\..
$$
Then $F_{m}$ is symmetric in $\bx=(x_1,x_2)$.
\end{lemma}

\begin{figure}[hbt]
\includegraphics[width=14.0cm]{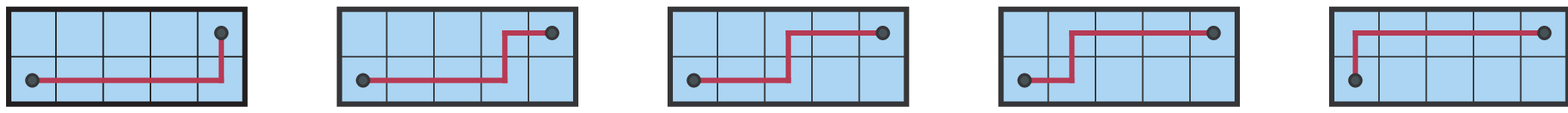}
\caption{Five paths~$\ga{}: (2,1)\to (1,5)$ in the $2$-symmetry Lemma~\ref{l:2sym}. }
\label{f:2sym}
\end{figure}

\begin{proof}[Proof]
There are $m$ paths in this case, see Figure~\ref{f:2sym}.  We have:
$$
F_{m} \, = \, G_m(x_1,x_2 \ts | \ts y_1,\ldots,y_m) \, \prod_{i=1}^2 \prod_{j=1}^m
\. \frac{1}{x_i+y_j}\,,
$$
where
$$G_m \, = \, (x_1+y_1)\cdots (x_1+y_{m-1}) \ts + \ts
(x_1+y_1)\cdots (x_1+y_{m-2})\cdot (x_2+y_m)  \ts + \ldots  + \ts
(x_2+y_2)\cdots (x_2+y_{m}).
$$
The symmetry of $G_m$ with respect to $x_1,x_2$ follows from the identity
$$(\diamond) \quad\qquad G_m=\frac{(x_1+y_1)(x_1+y_2)\ldots (x_1+y_m)-(x_2+y_1)(x_2+y_2)\ldots (x_2+y_m)}{x_1-x_2}\,.
$$
Indeed, the identity $(\diamond)$ can be proved by a telescopic cancellation:
\begin{align*}
G_m\cdot (x_1-x_2)\, &= \, (x_1+y_1)(x_1+y_2)\ldots (x_1+y_{m-1})\bigl[(x_1+y_m)-(x_2+y_m)\bigr]\\
& \qquad + \, (x_1+y_1)\cdots (x_1+y_{m-2})(x_2+y_m)\bigl[(x_1+y_{m-1})-(x_2+y_{m-1})\bigr] \\
& \qquad + \ \ldots \ +\, (x_2+y_2)(x_2+y_3)\cdots (x_2+y_{m})\bigl[(x_1+y_1)-(x_2+y_1)\bigr]\\
& = \, (x_1+y_1)(x_1+y_2)\ldots (x_1+y_m) \. - \. (x_2+y_1)(x_2+y_2)\ldots(x_2+y_m)\ts.
\end{align*}
Another way to prove~$(\diamond)$ is to note that both
parts are multilinear polynomials with respect to $y_1,\ldots,y_m$
and to check that they agree when $y_i\in\{-x_1,-x_2\}$ for all $i$.
\end{proof}

\smallskip

\subsection{Proof of Theorem~\ref{t:hor-cut}}
It suffices to show that $F_{\ca,\cb}$ is symmetric in \ts $(x_i,x_{i+1})$, for all
$1\le i<m$.  Fix a collection of paths $\Ups$ and consider only rows $i$ and~$(i+1)$.
Remove all columns where both squares are in $\Ups$ but not connected by a path, and
those columns where both squares are empty.
This results in several $2$-row rectangles, each connected
by a path from lower left corner to upper right corner.  Apply the $2$-symmetry
lemma to each non-empty rectangle to conclude that the sum of all $w(\Ups)$ is
symmetric in $(x_i,x_{i+1})$, as desired. \ $\sq$

\begin{figure}[hbt]
\psfrag{A1}{\small $A_1$}
\psfrag{A2}{\small $A_2$}
\psfrag{A3}{\small $A_3$}
\psfrag{A4}{\small $A_4$}
\psfrag{B1}{\small $B_1$}
\psfrag{B2}{\small $B_2$}
\psfrag{B3}{\small $B_3$}
\psfrag{B4}{\small $B_4$}
\psfrag{x1}{$x_1$}
\psfrag{x2}{$x_2$}
\psfrag{x3}{$x_3$}
\psfrag{x4}{$x_4$}
\epsfig{file=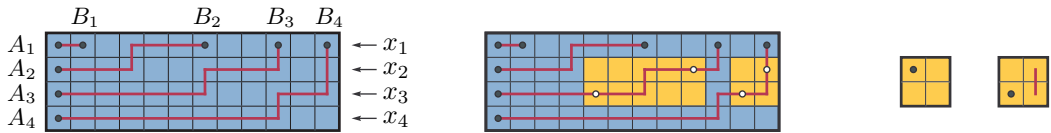,width=14.2cm}
\caption{\underline{Left}: Using $2$-symmetry to prove the symmetry of
$F_{\ca,\cb}$ in $(x_2,x_3)$. \ts \underline{Right}: Two impossible configurations. }
\label{f:ind-sym}
\end{figure}

\begin{rem}{\rm }
The proof above implicitly uses the claim
that $\ca$ and $\cb$ are as in the theorem.  Indeed, otherwise
we can have e.g.\ a rectangle with upper left square in $\ca$ and
no point of $\ca$ below it, or a point in $\cb$ in the bottom row
without the point of $\cb$ above it
(see Figure~\ref{f:ind-sym}).
\end{rem}


\smallskip

\subsection{Proof of Theorem~\ref{t:vert-cut}}
We follow the proof of Theorem~\ref{t:hor-cut} given above.  First, switch the
coordinates $\bx \lra \by$.  Then $\cb$ is as in Theorem~\ref{t:hor-cut}, while
$\ca$ are on the bottom row.  We need to prove the $\bx$-symmetry in this case.
Apply the $2$-symmetries in $(x_i,x_{i+1})$ in exactly the same way and notice
that the forbidden configuration as in the remark above do no appear.  The
details are straightforward. \ $\sq$

\smallskip

\subsection{The ultimate generalization} \label{ss:comb-gen}
The proofs above suggest a common generalization of
Theorems~\ref{t:hor-cut} and~\ref{t:vert-cut}.  We chose to
postpone it until this point to avoid overwhelming the reader.

\begin{thm}[Main theorem] \label{t:main-gen}
Let \ts $m,n,k\ge 1$, \ts $\bba=(a_1,\ldots,a_n)$,
$\bbb=(b_1,\ldots,b_n)$, where
$$
a_1+\ldots + a_n \. = \. b_1+\ldots + b_n \. = \. k\., \quad
\text{where} \ \. 0\le a_i,b_i \le m\ts.
$$
Let $\ca$ be a collection of points \ts $A_1,\ldots,A_k \in [m\times n]$,
with exactly $a_i$ points on the \emph{bottom} of $i$-th column.
Similarly, let $\cb$ be a collection of points \ts $B_1,\ldots, B_k \in [m\times n]$,
with exactly $b_i$ points on the \emph{top} of $i$-th column.
Here the order $\ca$ and $\cb$ is from left to right,
and within a column from top to bottom, see Figure~\ref{f:gen}.
Then the multivariate function
$$
F_{\ca,\cb}\bigl(x_1,\ldots,x_m\, | \, y_1,\ldots,y_n\bigr),
$$
defined in~\eqref{eq:F-def}, is symmetric in \ts $\bx=(x_1,\ldots,x_{m})$.
\end{thm}

\begin{figure}[hbt]
\psfrag{A1}{\footnotesize $A_1$}
\psfrag{A2}{\footnotesize $A_2$}
\psfrag{A3}{\footnotesize $A_3$}
\psfrag{A4}{\footnotesize $A_4$}
\psfrag{A5}{\footnotesize $A_5$}
\psfrag{A6}{\footnotesize $A_6$}
\psfrag{A7}{\footnotesize $A_7$}
\psfrag{B1}{\footnotesize $B_1$}
\psfrag{B2}{\footnotesize $B_2$}
\psfrag{B3}{\footnotesize $B_3$}
\psfrag{B4}{\footnotesize $B_4$}
\psfrag{B5}{\footnotesize $B_5$}
\psfrag{B6}{\footnotesize $B_6$}
\psfrag{B7}{\footnotesize $B_7$}
\epsfig{file=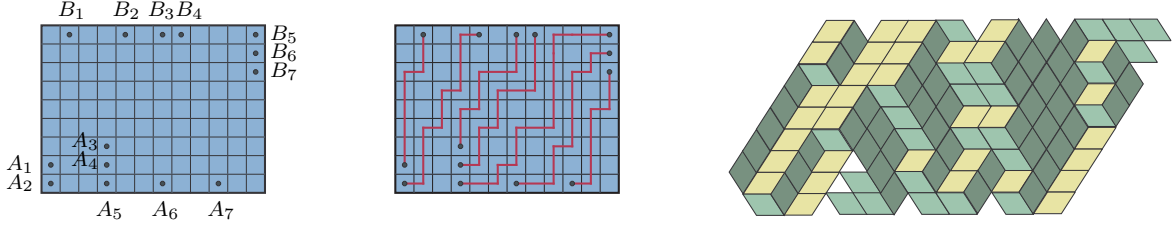,width=15.75cm}
\caption{\underline{Left}: Examples of a collection of points $\ca, \cb$, and
non-intersecting paths in Theorem~\ref{t:main-gen},
with \ts $\bba=(2,0,0,3,0,0,1,0,0,1,0,0)$ \ts and \ts $\bbb=(0,1,0,0,1,0,1,1,0,0,0,3)$. \ts
\underline{Right}: The corresponding lozenge tiling. }
\label{f:gen}
\end{figure}

The theorem generalizes Theorem~\ref{t:hor-cut} in a straightforward
way: take \ts $\bba=(k,0,\ldots,0)$ \ts and \ts $\bbb \in \{0,1\}^n$,
with $k$ zeroes.  It also generalizes Theorem~\ref{t:vert-cut} as follows: switch
coordinates \ts $\bx\lra \by$, and take both \ts
$\bba,\bbb \in \{0,1\}^n$, with $k$ zeroes. Of course, Theorem~\ref{t:main-gen}
is much more general, even if in some cases the result is vacuously true,
as there are no possible collections of non-intersecting \ts {\tt Up-Right}
paths \ts $\ga_i: A_i\to B_i$\ts.

\smallskip

\begin{proof}[Proof of Theorem~\ref{t:main-gen}]
The proof follows verbatim the proof of Theorem~\ref{t:hor-cut} given above.
We prove the $\bx$-symmetry via $2$-symmetries in $(x_i,x_{i+1})$
in exactly the same way.  Indeed, notice that the forbidden
configuration as in the remark above do no appear.
The details are straightforward.
\end{proof}

\bigskip

\section{Algebraic proofs}\label{s:alg}

\subsection{Preliminaries} \label{ss:alg-pre}
Fix $m,n \ge 1$ and let
$$
P_k(t) \. := \. (t+y_1)(t+y_2)\cdots (t+y_k), \ \ k=0,\ldots,n.
$$
For \ts $s=1,\ldots, m$, and \ts $k=1,\ldots, n$, define
$$
Q_{s,k}(t) \, := \, \prod_{j=s}^m \. \frac{1}{x_j-t} \, \mod \, P_k(t)\ts.
$$
Note that this expression is well defined: the polynomials $(x_j-t)$
are invertible modulo $P_k(t)$ in the ring $\cR[t]$, where $\cR=\cc(\bx,\by)$.
In other words, $\cR[t]$ is the ring of polynomials in~$t$ with
coefficients in the field of rational functions in $x_i$'s and $y_j$'s.

Denote
$$F_{s,k}\. := \. \sum_{\gamma:\ts (m,1)\to (s,k)} \. w(\gamma)\ts.
$$
We use the following description of $F_{s,k}$ which simultaneously proves
a $\bx$-symmetry and $\by$-symmetry of $F_{s,k}$.
This is the $k=1$ case of Theorem~\ref{t:MPP3-identity} generalizing
Lemma~\ref{l:2sym} to all~$m\ge 2$ (see also~$\S$\ref{ss:finrem-2sym}).

\begin{lemma}\label{l:path} For \ts $s=1,\ldots,m$ and \ts $k=1,\ldots,n$, we have:
$$
F_{s,k}\. = \. [t^{k-1}] \ts Q_{s,k}(t)\ts.
$$
In particular $F_{s,k}$ is symmetric with respect to \ts $(x_s,\ldots,x_m)$,
and with respect to \ts $(y_1,\ldots,y_k)$.
\end{lemma}

\begin{proof}
By definition,
$$F_{m,1} \. = \. \frac 1{x_m+y_1} \. = \. Q_{m,1}(t).
$$
Observe that
$$
F_{s,k} \. = \. \frac1{x_s+y_k} \. \bigl(F_{s,k-1}\ts +\ts F_{s+1,k}\bigr),
$$
for \ts $s=1,\ldots,m$, and \ts $k=1,\ldots,n$, such that \ \ts $(s,k)\ne (m,1)$.
Here we use boundary values \ts $F_{m+1,k}=F_{s,0}=0$.
Note that
$$
(t+y_k) \ts Q_{s,k}(t) \. \equiv \.  (t+y_k) \ts Q_{s,k-1}(t) \mod P_{k-1}(t) \ \ \. \text{and}
\mod (t+y_k).
$$
Thus, the congruence holds modulo $P_k(t)$~:
$$
(t+y_k) \ts Q_{s,k}(t) \. \equiv \.  (t+y_k) \ts Q_{s,k-1}(t) \. \mod \. P_{k}(t).
$$
Similarly,
$$
(x_s-t) \ts Q_{s,k}(t) \. = \.  Q_{s+1,k}(t) \. \mod  \. P_k(t).
$$
Adding these two congruences, we obtain
$$
(x_s+y_k) \ts Q_{s,k}(t) \. \equiv \.  Q_{s+1, k}(t)\ts + \ts (t+y_k) \ts Q_{s,k-1}(t) \.  \mod \. P_k(t).
$$
Now observe that both the LHS and the RHS are polynomials of degree at most $(k-1)$ in $t$.
Thus we have an equation of polynomials:
$$
(x_s+y_k) \ts Q_{s,k}(t) \. = \.  Q_{s+1, k}(t) \ts + \ts (t+y_k) \ts Q_{s,k-1}(t).
$$
Taking the coefficients of $t^{k-1}$, we see that the double
sequence
$$\bigl\{[t^{k-1}] \ts Q_{s,k}\bigr\}
$$
satisfies the same recurrence and initial conditions as $F_{s,k}$. This implies the result.
\end{proof}

\smallskip

\subsection{Non-intersecting paths} \label{ss:alg-LGV}
We recall the \emph{Lindstr\"om–-Gessel–-Viennot lemma}:

\begin{thm}[{LGV--lemma}] \label{t:LGV-e}
Let $G=(V,E)$ be a finite acyclic directed graph. 
Fix \ts $k\ge 1$.
Let \ts $\ca=\{A_1,\ldots,A_k\}$, \ts $\cb=\{B_1,\ldots,B_k\}\subset V$ \ts
be two $($not necessarily disjoint$)$ sets of vertices, such that  $|\ca|=|\cb|=k$.
Let $\rR$ be a commutative ring, and let \ts $w: E \to \rR$ \ts be a
weight function. For a subset $S\subset E$, define a weight
$$
w(S) \. : = \. \prod_{e\in S} \. w(e), \quad \text{and} \ \ \. w(\emp)\ts :=\ts 1\ts.
$$
Consider a matrix $U=(u_{ij})_{i,j=1}^k$, where
$$
u_{ij} \. := \. \sum_{\gamma : \ts A_i\to B_j} \. w(\gamma)
$$
is the sum of weights of all paths from $A_i$ to $B_j$. Then:
$$
\det U \, = \, \sum_{\pi\in S_k} \, \,
\sum_{\substack{\Ups=(\ga_1,\ldots,\ga_k) \\ \ga_i\ts{}:\ts{}
A_i\to B_{\pi(i)}}} \. \sign(\pi) \cdot w(\Ups),
$$
where the second sum is over 
all collections of vertex-disjoint paths~$\ga_i$ from $A_i$ to $B_{\pi(i)}$.
\end{thm}

For the proof, see~\cite[$\S$5.4]{GJ}, or~\cite{Tal} for a more general result.
Below, we will use the following ``vertex version'' of the LGV--lemma,
which easily follows from the above edge version.
In this corollary, a path is defined to be a set of vertices.

\begin{cor}[vertex--LGV]\label{c:LGV-v}
Let $G=(V,E)$ be a finite acyclic directed graph without multiple edges.
Fix \ts $k\ge 1$.
Let \ts $\ca=\{A_1,\ldots,A_k\},\cb=\{B_1,\ldots,B_k\}\subset V$ be two
(not necessarily disjoint) sets of vertices, such that \ $|\ca|=|\cb|=k$.
Let $\rR$ be a commutative ring, and let \ts $w: V \to \rR$ \ts  be a
weight function. For a subset $D\ssu V$, define a weight
$$
w(D) \. := \. \prod_{v\in D} \. w(v), \quad \text{and} \ \ \. w(\emp)\ts :=\ts 1\ts.
$$
Consider a matrix $U=(u_{ij})_{i,j=1}^k$, where
$$
u_{ij} \. : = \. \sum_{\gamma:\ts A_i\to B_j} \. w(\gamma)
$$
is the sum of weights of all paths $\ga$ from $A_i$ to $B_j$. Then
$$
\det U \, = \,
\sum_{\pi\in S_k} \sum_{\substack{\Ups=(\ga_1, \ldots,\ga_k) \\ \ga_i\ts{}:\ts{}
A_i\to B_{\pi(i)}}} \. \sign(\pi) \cdot w(\Ups),
$$
where the sum is over all 
collections of vertex-disjoint paths~$\ga_i$ from $A_i$ to $B_{\pi(i)}$.
\end{cor}

\begin{proof}
Denote by \ts $\wh G=(\wh V, \wh E)$ \ts the graph $G=(V,E)$ with added new vertices
$C_1,\ldots,C_k$ and directed edges $(C_iA_i)$.
For each edge $(XY)\in \wh E$, define its weight by $w(XY):=w(Y)$.
Apply Theorem~\ref{t:LGV-e} for the sets \ts $\cC=\{C_1,\ldots,C_k\}$ \ts
and \ts $\cb=\{B_1,\ldots,B_k\}$. Observe that the weight
of each path \ts $\bigl(C_iA_iX_1X_2\ldots X_nB_j\bigr)$ in~$\wh G$
is the same as the weight of the path \ts $\bigl(A_iX_1\ldots X_nB_j\bigr)$ \ts
in graph~$G$.  Similarly, the collections of
vertex-disjoint paths from $\cC$ to $\cb$ in~$\wh G$
are in a natural correspondence with
collections of vertex-disjoint paths from $\ca$ to $\cb$ in~$G$.
This implies the result.
\end{proof}

In many applications of the LGV--lemma, there is a unique permutation $\pi$ for
which there exists a vertex-disjoint collection of paths, and this unique $\pi$ is the identical
permutation, and the determinant equals to the weighted sum over
collections of disjoint paths from $A_i$ to $B_i$.
This also holds in the settings of Theorems~\ref{t:MPP3-identity},~\ref{t:hor-cut},~\ref{t:vert-cut} and
\ref{t:main-gen}.

\subsection{Proof of Theorem~\ref{t:vert-cut}}
By the vertex version of the LGV--lemma in Corollary~\ref{c:LGV-v}, the multivariate
rational function \ts
$F_{\ca,\cb}(\bx,\by)$ \ts
is a determinant of a $k\times k$ matrix~$U$ in which every entry $u_{ij}$ is
a rational function.  By Lemma~\ref{l:path}, these functions~$u_{ij}$ are
$\by$-symmetric. Thus the determinant is also $\by$-symmetric,
which completes the proof. \ $\sq$

\subsection{Proof of Theorem~\ref{t:hor-cut}}
In notation of Subsection~\ref{ss:alg-pre},
let \ts $\cR=\cc(\bx,\by)$. 
 Let $\mathcal{I}$ be the ideal generated
by the polynomials $P_{b_1}(t_1)$, $P_{b_2}(t_2),\ldots,P_{b_k}(t_k)$.
Consider the ring $\rR=\cR[t_1,\ldots,t_k]/\mathcal{I}$;
each element of this ring corresponds to a unique
polynomial $H(t_1,\ldots,t_k)$ with degrees less than $b_j$ in the variable~$t_j$,
for all $j=1,\ldots,k$.  For the elements of $\rR$, this allows us to define
the coefficients of the monomials \ts $t_1^{s_1}\ldots t_k^{s_k}$, where
$0\leqslant s_i<b_i$.

By the vertex version of the LGV--lemma in Corollary~\ref{c:LGV-v}, we have:
$$
F_{\ca,\cb} \. = \. \det\bigl(F_{A_i,B_j}\bigr)_{i,j=1}^k\, .
$$
Denote
$$\varphi_i(t_j) \. := \. (x_{a+i+1}-t_j)(x_{a+i+2}-t_j)\cdots (x_{a+k}-t_j)\ts,
$$
and observe the Vandermonde-type determinant
$$
(\circledast) \qquad
\det(\varphi_i(t_j))_{i, j=1}^k \, = \, \prod_{1\leqslant i<j\leqslant k} \. (t_j-t_i)\ts.
$$
The proof of~$(\circledast)$ follows the same argument as the standard proof of
the (usual) Vandermonde determinant formula.

The elements \ts $(t_i-x_j)$ \ts are invertible in $\rR$,
and by Lemma~\ref{l:path} we have:
$$
F_{A_i,B_j} \, = \, \Bigl[t_j^{b_j-1}\Bigr] \. \prod_{s=1}^{a+i} \. \frac1{x_s-t_j}\,= \,
\Bigl[t_j^{b_j-1}\Bigr] \, \varphi_i(t_j) \prod_{s=1}^{m} \frac{1}{x_s-t_j}\,.
$$
Interchanging the coefficients-evaluating functional and the determinant sign and
applying~$(\circledast)$, we obtain:
$$\aligned
F_{\ca,\cb} \, & = \, \left[t_1^{b_1-1}\ldots t_k^{b_k-1}\right]
\, \prod_{j=1}^k\prod_{s=1}^m \. \frac1{x_s-t_j} \. \det\bigl(\varphi_i(t_j)\bigr)_{i, j=1}^k\\
& = \, \left[t_1^{b_1-1}\ldots t_k^{b_k-1}\right]  \prod_{1\leqslant i<j\leqslant k} \. (t_j-t_i) \, \prod_{j=1}^k\prod_{s=1}^m\frac{1}{x_s-t_j}\,.
\endaligned
$$
The RHS is certainly symmetric in \ts $\bx=(x_1,\ldots,x_m)$. This completes the proof
of the theorem. \ $\sq$


\bigskip

\section{Final remarks}\label{s:finrem}

\subsection{Many hidden symmetries}
\label{ss:finrem-other}
As we mentioned in the introduction, hidden symmetries are
a staple in Algebraic and Enumerative Combinatorics. Without
aiming to review even a fraction of the literature,
let us mention a few notable examples.
First, the \emph{Littlewood--Richardson coefficients}
have a number of hidden symmetries not reflected in their
classical combinatorial interpretation.  While the
\emph{BZ-triangles}~\cite{BZ} combined with bijections
in~\cite{PV1} explained some of the symmetries, others
remain unexplained, see~\cite[$\S$6.6]{PV2}.

Another major appearance of the hidden symmetries is
in connection with the \emph{alternating sign matrices},
which led to a conceptual proof by Kuperberg~\cite{Kup}.
Further symmetries of ASMs were discovered by
Razumov--Stroganov~\cite{RS} (see also~\cite{Wie}),
and eventually proved by a technical argument in~\cite{CS}.

Finally, in a fascinating study (completely unrelated
to this work), Coxeter used the symmetry of regular solids in
$\rr^4$ to evaluate special values of the \emph{dilogarithm}~\cite{Cox}.
The following amazing identity coming from the \emph{600-cell}
is a testament to the power of hidden symmetries:
$$
\sum_{n=1}^\infty \,\. \frac{\phi^n}{n^2} \, \cos\left(\frac{2\pi n}{5}\right) \, = \, \frac{\,\pi^2}{100}
\,, \quad
\text{where} \quad \phi \. = \. \frac{\sqrt{5}-1}{2}\..
$$

\subsection{Yang--Baxter equations}  \label{ss:finrem-BP}
Closer to the subject, Borodin in~\cite{Bor}
initiated the study of symmetric rational functions for
the \emph{six-vertex model} which are proved via the
\emph{Yang--Baxter equations}, see~\cite{Bax}.
These results were greatly extended in~\cite{BP1}
(see also a survey~\cite{BP2}).  These functions have
multiple families of parameters, but they do not
specialize to our functions \ts $F_{\ca,\cb}(\bx,\by)$.
To see this, note that in our setting, the intersections
are not allowed, making it a \emph{five-vertex model},
implying degeneration of many parameters.

In a parallel investigation, Bump, McNamara and Nakasuji~\cite{BMN}
realized that the factorial Schur functions can be expressed
as the partition function of a six-vertex model with certain
particular multivariate parameters.  When \ts $t=-1$, the
deformation in~$\S$4 in their paper gives new solutions of
the Yang--Baxter equations exactly with the same parameters as
are implicit in this paper.  In particular, this gives a new
proof of the 2-symmetry in Lemma~\ref{l:2sym}, the fourth proof
counting two proofs in this paper and one in~\cite{MPP3}, but
perhaps the most conceptual one.  We learned about~\cite{BMN}
only after this paper was written.

We should emphasize that a solution of the Yang--Baxter equations
is not enough to establish the symmetry, as one needs to check the
boundary conditions.  This is what makes our Main Theorem~\ref{t:main-gen}
so surprising -- it gives the most unusual boundary conditions for
which the symmetry holds.

\subsection{Further symmetries}  \label{ss:finrem-shift}
Let us mention some recent progress in this setting,
the \emph{shift invariance} for the six-vertex model
and \emph{polymers}, discovered recently in~\cite{BGW}.
It can be viewed as the new fundamental (multivariate)
hidden symmetry for the number of certain lattice
path configurations.  This shift invariance found a surprising
application in~\cite{BGR1} to certain properties of multi-particle
generalization of TASEP, in turn related to the number of
reduced factorizations of certain permutations in~$S_n$.
Most recently, \cite{Gal} established a more general type of symmetries
called \emph{flip invariance}, and gave them a conceptual
algebraic explanation.  In a different direction,
curious combinatorial implications of this
and related symmetries were found in~\cite{Dau}.

\subsection{Factorial Schur functions} In notation of~$\S$\ref{ss:main-lozenge},
when $a=0$ as in Figure~\ref{f:Sawtooth},
one can think of $F_{\ca,\cb}(\bx, \ts | \ts \by)$ in
Theorem~\ref{t:hor-cut} the multivariate deformation
of \ts $N(\la)$.  This deformation is different, but curiously
similar to the $\bx$-symmetric and $\by$-parametrized
factorial Schur functions \ts $s_\la(\bx\ts | \ts -\by)$,
which forms a basis in symmetric polynomials of~$\bx$,
see~\cite[$\S$6]{Mac}.  This should not come as a surprise
as  the proof in~\cite{MPP3} is based on combinatorics and
algebra of factorial Schur functions.  It would be
interesting to establish a formal connection in
full generality.

\subsection{Selberg integral}\label{ss:finrem-KO}
In~\cite{KO}, the authors proved some of the corollaries of~\cite{MPP3}.
The results follow from the \emph{Selberg integral}, another yet to be fully
understood hidden symmetry, see \cite[$\S$9.3-9.4]{MPP3}.

\subsection{Identities}\label{ss:finrem-MPP3}
Theorem~\ref{t:MPP3-identity} is stated in \cite[Thm 3.10]{MPP3}
in a weaker form, but the result follows from the proof.  Of course,
we both reprove and generalize it in this paper.  Note, however,
that Thm~3.12 in the same paper gives a different hidden symmetry
which does not follow from this paper.

\subsection{Evaluations}\label{ss:finrem-2sym}
As suggested by both our combinatorial and
algebraic proofs, Theorem~\ref{t:MPP3-identity} is not obvious
already for~$k=1$. Even the special case $k=1$, $x_i=i$ and
$y_j=b-j+1$, is already quite interesting~\cite[Cor.~3.11]{MPP3}.

\subsection{Generalizations}\label{ss:finrem-LA}
The combinatorial proof in Section~\ref{s:comb} may appear to be
more flexible, as it leads to the proof of our Main
Theorem~\ref{t:main-gen}.  However, the algebraic proofs
tend to be more powerful and amenable to generalizations of
different kind.  For example, it would be interesting if
the results generalize to three and higher dimensions as we
seem to have exhausted the planar version.
In a different direction, the determinant style proofs
as in Section~\ref{s:alg}, suggest
possibility of non-commutative generalization, cf.~\cite{GR}.
Finding a proper $q$-analogue (or quantum analogue?) would
be especially interesting.

\subsection{Up-Right condition}\label{ss:finrem-UR}
Theorem~\ref{t:vert-cut} remains true even when the assumption
that all paths are required to be \ts {\tt Up-Right} \ts is removed.
This leads to a somewhat stronger but less natural result.
We leave the proof to the reader.  Let us note, however,
that while the \ts {\tt Up-Right} \ts condition is vacuous
for Theorem~\ref{t:hor-cut}, it is necessary for our Main
Theorem~\ref{t:main-gen}.

\bigskip

{\small
\subsection*{Acknowledgements}
We are grateful to Alejandro Morales and Greta Panova for many
interesting conversations about the Naruse hook-length formula over
the years.  Special thanks to Vadim Gorin and Leo Petrov for telling
us about connections to the six-vertex model and help with the
references, to Alexey Borodin for telling us about~\cite{BMN},
and to Darij Grinberg for numerous helpful remarks.

These results were obtained during the \emph{Asymptotic
Algebraic Combinatorics} workshop at the IPAM; both authors thank
IPAM for the hospitality, organization and inspiration.
The first author was partially supported by the NSF. The second
author was partially supported by the Foundation for the
Advancement of Theoretical Physics and Mathematics ``BASIS''.

\end{document}